\documentclass[12pt, a4paper, leqno]{amsart}
\usepackage[utf8]{inputenc}
\usepackage{amsmath}
\usepackage{amsfonts}
\usepackage{amssymb}
\usepackage{times}
\usepackage[T1]{fontenc} 
\usepackage{url} 
\usepackage{color,esint}
\usepackage[dvipsnames]{xcolor}
\usepackage[colorlinks, allcolors=RedViolet,pdfstartview=,pdfpagemode=UseNone]{hyperref} 
\usepackage{graphicx} 
\usepackage{enumitem}
\usepackage{tabularx}
 \usepackage{mathtools}
\usepackage{pgf,tikz}
\usepackage{mathrsfs}
\usetikzlibrary{arrows}

\let\oldmarginpar\marginpar
\renewcommand\marginpar[1]{\-\oldmarginpar[\raggedleft\tiny #1]%
{\raggedright\tiny #1}}

\usepackage{amsthm}
\theoremstyle{plain}
\newtheorem{thm}[equation]{Theorem}
\newtheorem{lem}[equation]{Lemma}
\newtheorem{prop}[equation]{Proposition}
\newtheorem{cor}[equation]{Corollary}

\theoremstyle{definition}
\newtheorem{defn}[equation]{Definition}

\newtheorem{eg}[equation]{Example}

\theoremstyle{remark}
\newtheorem{rem}[equation]{Remark}

\numberwithin{equation}{section}

\newcommand{\R}{\mathbb{R}}

\newcommand{\Rn}{{\mathbb{R}^n}}

\def\weakto{\rightharpoonup}

\renewcommand{\div}{\divop}
\renewcommand{\phi}{\varphi}
\renewcommand{\epsilon}{\varepsilon}
\def\le{\leqslant}
\def\leq{\leqslant}
\def\ge{\geqslant}

\def\phi{\varphi}
\def\rho{\varrho}
\def\vartheta{\theta}

\def\I{\mathcal{I}}

\def\div{\qopname\relax o{div}}

\def\loc{{\rm loc}}
\def\BV{{\rm BV}}

\newcommand{\inc}[1]{\hyperref[def:aInc]{{\normalfont(Inc){\ensuremath{_{#1}}}}}}
\newcommand{\dec}[1]{\hyperref[def:aDec]{{\normalfont(Dec){\ensuremath{_{#1}}}}}}
\newcommand{\ainc}[1]{\hyperref[def:aInc]{{\normalfont(aInc){\ensuremath{_{#1}}}}}}
\newcommand{\adec}[1]{\hyperref[def:aDec]{{\normalfont(aDec){\ensuremath{_{#1}}}}}}
\newcommand{\adeci}[1]{\hyperref[def:aDeci]{{\normalfont(aDec){\ensuremath{_{#1}^\infty}}}}}
\newcommand{\azero}{\hyperref[def:a0]{{\normalfont(A0)}}}
\newcommand{\aone}{\hyperref[def:a1]{{\normalfont(A1)}}}
\newcommand{\aones}[1]{\hyperref[def:a1s]{{\normalfont(A1-{\ensuremath{{#1}})}}}}

\newcommand{\Phiw}{\Phi_{\text{\rm w}}}
\newcommand{\Phic}{\Phi_{\text{\rm c}}}

\newcommand{\rholiminf}{\bar\rho^{\,\scriptscriptstyle\circ}_{1,\phi}}
\newcommand{\liminftag}{\overline{^{\scriptscriptstyle\circ}_{1,\phi}}}
\newcommand{\rhoBV}{\bar\rho_{1,\phi}}

\newcommand{\Lbar}[1]{{\bar L^{1, \phi}(#1)}}
\newcommand{\Lspace}[1]{{L^{1, \phi}(#1)}}

\date{\today}

\newcommand{\comment}[1]{\vskip.3cm
\fbox{%
\parbox{0.93\linewidth}{\footnotesize #1}}
\vskip.3cm}

\begin{document}

\title{Minimizers of abstract generalized Orlicz--bounded variation energy}

\date{\today}

\author{Michela Eleuteri}
\address{Michela Eleuteri, 
Dipartimento di Scienze Fisiche, Informatiche e Matematiche, 
Università degli Studi di Modena e Reggio Emilia, Italy}
\email{\texttt{michela.eleuteri@unimore.it}}

\author{Petteri Harjulehto}
\address{Petteri Harjulehto,
Department of Mathematics and Statistics,
FI-20014 University of Turku, Finland}
\email{\texttt{petteri.harjulehto@utu.fi}}

\author{Peter Hästö}
\address{Peter Hästö, Department of Mathematics and Statistics,
FI-20014 University of Turku, Finland}
\email{\texttt{peter.hasto@utu.fi}}

\begin{abstract}
A way to measure the lower growth rate of $\phi:\Omega\times [0,\infty) \to [0,\infty)$ is 
to require $t \mapsto \phi(x,t)t^{-r}$ to be increasing in $(0,\infty)$.
If this condition holds with $r=1$, then 
\[
\inf_{u\in f+W^{1, \phi}_0(\Omega)}\int_\Omega \phi(x, |\nabla u|) \, dx
\]
with boundary values $f\in W^{1,\phi}(\Omega)$ does not necessary have a minimizer. However, if 
$\phi$ is replaced by $\phi^p$, then the  
growth condition holds with $r=p > 1$ and thus (under some additional conditions) the corresponding energy integral has a minimizer. We show that a sequence $(u_p)$ of such minimizers convergences when $p \to 1^+$
in a suitable $\BV$-type space involving generalized Orlicz growth and obtain the 
$\Gamma$-convergence of functionals with fixed boundary values and of functionals with fidelity terms. 
\end{abstract}


\maketitle

\section{Introduction}

Since its introduction in the early 1970's, $\Gamma$-convergence has gained more and more importance, being a very flexible instrument and the most natural notion of convergence for variational problems. Therefore in the last decades much literature has been devoted to the description of the asymptotic behaviour of families of minimum problems, usually depending on some parameters, appearing in various contexts.

$\Gamma$-convergence is mostly applied to families of integral functionals, among which a key role is played by the $p$-Dirichlet integral, 
\[
\int_\Omega |\nabla u|^p\, dx,
\]
as a prototype for integral functionals with standard growth. The case $p = 1$ deserves particular attention: questions like semicontinuity and relaxation require the use of functions of bounded variation, $\BV(\Omega)$, 
instead of the space $W^{1,1}(\Omega)$ \cite{Dal80} unless some additional assumptions are made \cite{FM}. 
Recall that $\BV(\Omega)$ is the Banach space of all $L^1(\Omega)$-functions whose first order distributional derivatives are bounded Radon measures.

It is natural to ask if the functionals for $p\to 1^+$ are connected to the case $p=1$ . 
However, few results are presented in literature, at least to our knowledge.
Juutinen \cite{Juu05} studied the problem for so-called 
{\it functions of least gradient} in $\BV(\Omega) \cap C(\overline{\Omega})$ 
with given boundary values and continuous minimizers of the $p$-Dirichlet energy.
More precisely, for suitable domains $\Omega$, the sequence of the unique $p$-minimizers $u_p \in W^{1,p}_{\loc}(\Omega) \cap C(\overline{\Omega})$
with boundary value $w \in C(\partial \Omega)$ converges uniformly, as $p \to 1^+$, to a function $u \in \BV(\Omega) \cap C(\overline{\Omega})$ that is the unique function of least gradient with boundary value $w$.
Apart from the continuity requirement, these variational problems represent the fundamental minimization problems in their respective spaces $\BV$ and $W^{1,p}$.

The concept of $\Gamma$-convergence, introduced by De Giorgi and Franzoni \cite{DeGF75}, has been systematically 
presented in \cite{Bra02, Dal93}. 
$\Gamma$-convergence was not used by Juutinen \cite{Juu05}, but it seems reasonable to use it in this 
context, as well. 
A family of functionals $\I_\epsilon: X \to \overline{\R}$ is said to \textit{$\Gamma$-converge}
(in topology $\tau$) to $\I: X \to \overline{\R}$
if the following hold for every positive sequence $(\epsilon_i)$ converging to zero:
\begin{enumerate}
 \item[(a)] $\displaystyle \I (u) \le \liminf_{i \to \infty} \I_{\epsilon_i} (u_i)$ 
for every $u \in X$ and every $(u_i)\subset X$ $\tau$-converging to $u$;
 \item[(b)]
$\displaystyle \I (u) \ge \limsup_{i \to \infty} \I_{\epsilon_i} (u_i)$
for every $u \in X$ and some $(u_i)\subset X$ $\tau$-converging to $u$.
\end{enumerate}
We will use this definition for $\I_{\epsilon_i} = E_{1+\epsilon_i}$; for simplicity we denote 
$p_i = 1 + \epsilon_i$ and talk about the $\Gamma$-convergence of $E_{p_i}$ as $p_i\to 1^+$.

In this paper we consider the following functionals with generalized Orlicz growth. 
This is a very active field recently, boosted by work on the double phase problem 
by Baroni, Colombo and Mingione, e.g.\ \cite{BarCM18, ColM15a}. The generalized 
Orlicz (also known as Musielak--Orlicz) case unifies the study of the 
double phase problem and the variable exponent growth, which has been intensively studied 
in the last 20 years \cite{DieHHR11}. We mention as examples the following very recent 
studies \cite{BaaB22, BenHHK21, ChlK21, DefM21, MaeMO21, MizS_pp, PapRZ22, SkrV_pp, Ste22}. 
One motivating factor is applications to image processing initially proposed 
by Chen, Levine and Rao \cite{CheLR06}, see also \cite{HarH21, HarHL08}. 
In this context one is especially interested in the minimization problem with 
lower growth equal to $1$. Here we study such functionals in an abstract 
and general setting.

To state our 
results we use the notation from Section~\ref{sect:background} and refer to \cite{HarH_book} for more background. 
Let $\phi$ be a weak $\Phi$-function and $u_0 \in W^{1, \varphi}(\Omega)$ be the boundary value function.
In Section~\ref{sect:boundaryValues} we study the functional 
$E_p : L^1(\Omega)\to [0,\infty]$ with fixed boundary values
$u_0\in W^{1,\phi^r}(\Omega)$, where $r>1$, for $p\in (1,r)$ defined by 
\begin{equation}
\label{eq:Ep}
E_p(u) := 
\begin{cases}
\displaystyle\int_\Omega \phi(x, |\nabla u|)^p \, dx & \text{when } u - u_0\in L^{1, \phi^p}_{0}(\Omega); \\
+\infty & \text{otherwise.}
\end{cases}
\end{equation}
We define the limit functional $E: L^1 (\Omega) \to [0, \infty]$ by
\begin{equation*} 
E(u) := 
\inf\Big\{\liminf_{i\to\infty}\int_\Omega \phi(x, |\nabla u_i|) \, dx : u_i -u_0 \in L^{1, \phi}_{0}(\Omega)\text{ and } u_i\to u
\text{ in } L^1(\Omega) \Big\}.
\end{equation*}
In Section \ref{sect:fidelityTerm} we consider corresponding functionals with a fidelity term. We assume this time that $f \in L^2(\Omega)$.
For $p >1$ we define $F_p: L^2(\Omega) \to [0, \infty]$ by
\begin{equation}
\label{eq:Fp}
F_p(u) := 
\begin{cases}
\displaystyle\int_\Omega \phi(x, |\nabla u|)^p + |u-f|^2 \, dx & \text{when } u \in L^{1, \phi^p}(\Omega) \cap L^2 (\Omega); \\
+\infty & \text{otherwise.}
\end{cases}
\end{equation}
%
Then we define the limit functional $F: L^2 (\Omega) \to [0, \infty]$ by
\begin{equation*} 
F(u) := 
\inf\Big\{\liminf_{i\to\infty}\int_\Omega \phi(x, |\nabla u_i|) + |u_i-f|^2\, dx : 
u_i \in L^{1, \phi} (\Omega) \cap L^2(\Omega)\text{ and } u_i\to u \text{ in } L^2(\Omega) \Big\}
\end{equation*}
Our main result therefore is the following:

\begin{thm}\label{thm:Gamma}
Let $\Omega\subset\Rn$ be a bounded domain and let $\phi\in \Phiw(\Omega)$ satisfy \azero{} and \dec{}.
\begin{enumerate}
\item 
$E_p$ $\Gamma$-converges to $E$ in the topology of $L^1(\Omega)$ as $p \to 1^+$.
\item
$F_p$ $\Gamma$-converges to $F$ in the topology of $L^2(\Omega)$ as $p \to 1^+$ 
if $C^\infty(\overline\Omega)$ is dense in $L^{1,\phi}(\Omega)\cap L^2(\Omega)$.
\end{enumerate}
\end{thm}

Due to the fidelity term in $L^2$, it seems natural to consider the latter convergence in $L^2(\Omega)$, although other choices are also possible. From the proofs we see that  assumption \azero{} is only 
needed for the case $E(u)=\infty$ or $F(u)=\infty$; if we consider $\Gamma$-convergence on the 
space $X:=\{u\in L^1(\Omega) : E(u)<\infty\}$ or $X:=\{u\in L^2(\Omega) : F(u)<\infty\}$, then 
this assumption can be omitted. 

The proof of Theorem~\ref{thm:Gamma}(1) is presented in Section~\ref{sect:boundaryValues} 
along with results for minimizers of the energy $E$. The companion section, \ref{sect:fidelityTerm}, 
covers Theorem~\ref{thm:Gamma}(2) and minimizers of the energy $F$. The results for minimizers of the energies $E$ and $F$ generalize the corresponding ones in \cite{Juu05}.
Finally, in Section~\ref{sect:specialCases}, we give explicit formulas for $F$ in the 
important special cases $\phi(x,t) = t + a(x) t^2$ (the double phase case) and $\phi(x,t) = t^{p(x)}$ (the variable exponent case) based on \cite{HarH21,HarHLT13, HarHL08}.

We would like to remark that the limit energies $E$ and $F$ turn to be particular cases of the more general modular we introduce and study in Section \ref{sect:liminf}; this new definition is inspired by the work of
Miranda \cite{Miranda03} and covers classical bounded variation spaces and Sobolev spaces.


\section{Background}
\label{sect:background}

Throughout the paper we always consider a 
\textbf{bounded domain $\Omega \subset \Rn$}, i.e.\ a bounded, open and connected 
set. By $p':=\frac p {p-1}$ we denote the H\"older conjugate exponent 
of $p\in [1,\infty]$. The notation $f\lesssim g$ means that there exists a constant
$c>0$ such that $f\le c g$. The notation $f\approx g$ means that
$f\lesssim g\lesssim f$.
By $c$ we denote a generic constant whose
value may change between appearances.
A function $f$ is \textit{almost increasing} if there
exists $L \ge 1$ such that $f(s) \le L f(t)$ for all $s \le t$
(more precisely, $L$-almost increasing).
\textit{Almost decreasing} is defined analogously.
By \textit{increasing} we mean that this inequality holds for $L=1$ 
(some call this non-decreasing), similarly for \textit{decreasing}. 

\begin{defn}
\label{def2-1}
We say that $\phi: \Omega\times [0, \infty) \to [0, \infty]$ is a 
\textit{weak $\Phi$-function}, and write $\phi \in \Phiw(\Omega)$, if 
the following conditions hold for almost every $x\in\Omega$:
\begin{itemize}
\item 
For every measurable function $f:\Omega\to \R$ the function $y \mapsto \phi(y, f(y))$ is measurable.
\item
The function $t \mapsto \phi(x, t)$ is non-decreasing.
\item 
$\displaystyle \phi(x, 0) = \lim_{t \to 0^+} \phi(x,t) =0$ and $\displaystyle \lim_{t \to \infty}\phi(x,t)=\infty$.
\item 
The function $t \mapsto \frac{\phi(x, t)}t$ is $L$-almost increasing on $(0,\infty)$ with $L$ independent of $x$.
\end{itemize}
If $\phi\in\Phiw(\Omega)$ is additionally convex and left-continuous with respect to $t$ for almost every $x$, then $\phi$ is a 
\textit{convex $\Phi$-function}, and we write $\phi \in \Phic(\Omega)$. 
\end{defn}

For $\phi :\Omega\times [0,\infty)\to [0,\infty)$ and $p,q>0$ we define some conditions:
\begin{itemize}
\item[(A0)]\label{def:a0}
There exists $\beta \in(0, 1]$ such that $\phi(x, \beta) \le 1 \le \phi(x, \frac1\beta)$ 
for a.e.\ $x \in \Omega$. 

%

\item[(aInc)$_p$] \label{def:aInc} 
$t \mapsto \frac{\phi(x,t)}{t^{p}}$ is $L_p$-almost 
increasing in $(0,\infty)$ for some $L_p\ge 1$ and a.e.\ $x\in\Omega$.

\item[(aDec)$_q$] \label{def:aDec}
$t \mapsto \frac{\phi(x,t)}{t^{q}}$ is $L_q$-almost 
decreasing in $(0,\infty)$ for some $L_q\ge 1$ and a.e.\ $x\in\Omega$.
\end{itemize} 
%
%
We say that \ainc{} holds if \ainc{p} holds for some $p>1$, and similarly for \adec{}.
We say that $\phi$ satisfies \inc{p} if 
\ainc{p} holds with $L_p=1$, similarly for \dec{q}.
%
%

If $\phi\in \Phiw(\Omega)$, then $\phi(\cdot ,1)\approx 1$ implies \azero{}, and if $\phi$ satisfies \adec{}, then \azero{} and $\phi(\cdot ,1)\approx 1$ are equivalent. For instance, $\phi(x, t)=t^p$ always satisfies 
\azero{}, since $\phi(x, 1) \equiv 1$. 
Note that $\phi\in\Phiw(\Omega)$ includes the assumption that $\phi$ satisfies \ainc{1}.  
Finally, \ainc{} and \adec{} measure the lower and upper growth rates.


We recall some basic notions of generalized Orlicz spaces \cite{HarH_book}.

\begin{defn}
Let $\varphi \in \Phiw(\Omega)$ and $\displaystyle\rho_\phi(u) := \int_{\Omega} \varphi (x, |u|) \, dx$
for all measurable functions $u \in L^0(\Omega)$. The set
\[
L^{\varphi}(\Omega) := \{u \in L^0(\Omega): \rho_\phi(\lambda \, u) < \infty \quad \textnormal{for some $\lambda > 0$}\}
\]
is called a \textit{generalized Orlicz space}. 
We define a (quasi-)norm on this space by
\[
\|u\|_\phi := \inf \left \{\lambda > 0: \,\, \rho_\phi \left (\frac{u}{\lambda} \right ) \le \, 1\right \}.
\]
We use the abbreviation $\|v\|_\phi := \big\| |v|\big\|_\phi$ 
for vector-valued functions.
\end{defn}

We will often use the following lower semi-continuity result, for which we require the 
stronger, convex $\Phi$-function assumption. 

\begin{lem}[Theorem~2.2.8, \cite{DieHHR11}]\label{lem:lsc}
Let $\phi \in \Phic(\Omega)$. If $u_i \rightharpoonup u$ in $L^\phi(\Omega)$, then
\[
\int_\Omega \phi(x, |u|) \, dx \le \liminf_{i \to \infty} \int_\Omega \phi(x, |u_i|) \, dx,
\]
i.e.\ the modular $\rho_\phi$ is weakly (sequentially) lower semicontinuous. 
\end{lem}

%


\section{Abstract \texorpdfstring{$\BV$}{BV}-type spaces}
\label{sect:liminf}

We define a Sobolev-type space with function in $L^1$ and gradient in $L^\phi$: 
\[
L^{1, \phi} (\Omega) := \{u \in W^{1, 1} (\Omega) : \|u\|_{L^{1, \phi}(\Omega)} < \infty\}
\quad\text{where}\quad
\|u\|_{L^{1, \phi}(\Omega)} := \|u\|_{L^1(\Omega)} + \|\nabla u\|_{L^{\phi}(\Omega)}.
\]
Note that $L^{1, \phi}(\Omega)\subset W^{1, 1}(\Omega)$, and   $W^{1, \phi}(\Omega) = L^{1, \phi} (\Omega) \cap L^\phi(\Omega)$. 
In the next definition and the rest of the article we use the $^\circ$-symbol to indicate only the gradient part 
of a Sobolev-type norm and the $\bar{\phantom{m}}$-symbol to indicate the $\liminf$-approximation process. 

\begin{defn}\label{defn:norms}
For $\phi\in\Phiw(\Omega)$ and $u \in L^1(\Omega)$ and define
\[
\|u\|\liminftag := 
\inf\big\{\liminf_{i\to \infty}\|\nabla u_i\|_\phi : 
u_i \in \Lspace{\Omega}\text{ and } u_i \to u \text{ in } L^1(\Omega) \big\},
\]
\[
\rholiminf (u) := \inf\big\{\liminf_{i\to \infty}\rho_\phi(|\nabla u_i|) : 
u_i \in \Lspace {\Omega}\text{ and  }u_i \to u \text{ in } L^1(\Omega) \big\},
\]
\[
\rhoBV(u) := \rho_{1}(u) + \rholiminf(u) 
\qquad\text{and}\qquad
\Lbar{\Omega} := \{ u\in L^1(\Omega) : \rhoBV(\lambda u) <\infty \text{ for some } \lambda >0\}. 
\]
\end{defn}

By a diagonal argument we find, for every $u\in L^1(\Omega),$ functions  
$u_i \in \Lspace{\Omega}$ with $u_i \to u$ in $L^1(\Omega)$ and 
\[
\rholiminf (u) = \lim_{i\to \infty}\rho_\phi(|\nabla u_i|).
\]
We call this a test sequence or a sequence that gives $\rholiminf$. 
The same idea works for the norm $\|u\|\liminftag$. Let us first motivate our definitions with two examples 
which show that we cover the spaces $W^{1,\phi}(\Omega)$ and $\BV(\Omega)$. 

\begin{eg}[Bounded variation spaces]
If $\phi(x, t) \equiv t$, then $\rholiminf(u) = V(u, \Omega)$, the total variation of $u$, 
for every $u \in \BV(\Omega)$.
Indeed, by \cite[Theorem 5.3]{EvaG92} or \cite[Theorem 3.9]{AmbFP00} there exist $u_i \in C^\infty(\Omega)$ such that $u_i \to u$ in $L^1(\Omega)$ and
\[
V(u, \Omega) = \lim_{i \to \infty} \int_\Omega |\nabla u_i| \, dx.
\]
Thus, by the definition of $\rholiminf(u)$, we obtain that
\[
\rholiminf(u) \le \lim_{i \to \infty} \int_\Omega |\nabla u_i| \, dx = V(u, \Omega).
\]
On the other hand, let $(u_i)$ give $\rholiminf(u)$. Since $u_i \to u$ in $L^1(\Omega)$, we obtain by the lower semicontinuity of the variational measure \cite[Theorem 5.2]{EvaG92} that
\[
V(u, \Omega) \le \liminf_{i \to \infty} V(u_i, \Omega) = \liminf_{i \to \infty} \int_\Omega |\nabla u_i| \, dx = \rholiminf(u). 
\]
\end{eg}

\begin{eg}[Sobolev spaces]
If $\phi\in\Phic(\Omega)$, $C^\infty_0(\Omega)$ is dense in the dual space $(L^{\phi}(\Omega))^*$ and $u\in \Lspace\Omega$, then 
$\rholiminf (u) = \rho_\phi(|\nabla u|)$. First of all, since 
$u \in \Lspace{\Omega}$, we can choose the constant sequence $u_i=u$ so that 
$\rholiminf(u) \le \rho_\phi(|\nabla u|)$ by the definition. 

For the opposite inequality we choose $u_i\in \Lspace{\Omega}$ with $u_i\to u$ in $L^1(\Omega)$. 
Let $g\in C^\infty_0(\Omega; \Rn)$. 
By the definition of the weak derivative and the $L^1$-convergence, 
\[
\int_\Omega \nabla u\cdot g \, dx 
=
-\int_\Omega u\div g \, dx 
=
-\int_\Omega u_i\div g \, dx 
=
\int_\Omega \nabla u_i\cdot g \, dx. 
\]
Since $C^\infty_0(\Omega)$ is dense in $(L^{\phi}(\Omega))^*$ we obtain the claim for 
any $g\in (L^{\phi}(\Omega; \Rn))^*$ (with component-wise approximation), so $\nabla u_i\rightharpoonup \nabla u$ and also $|\nabla u_i|\rightharpoonup |\nabla u|$ in $L^\phi$. 
Thus it follows from Lemma~\ref{lem:lsc} that 
\[
\rho_\phi(|\nabla u|) \le \liminf \rho_\phi(|\nabla u_i|).
\]
Taking the infimum over such sequences $(u_i)$, we 
obtain the opposite inequality, $\rho_\phi(|\nabla u|)\le \rholiminf(u)$.
\end{eg}

%

Let us then consider $\rholiminf$. 
Note that $u \mapsto \rholiminf(u)$ is not a semimodular on $L^1(\Omega)$ 
in the sense of \cite[Definition~2.1.1]{DieHHR11} since $\rholiminf(c)=0$ for every constant $c\in \R$. Indeed, as we saw in the 
example, $\rholiminf$ is like $\rho_\phi(|\nabla u|)$, so it is natural to add the modular 
of the function to it, otherwise we can only expect it to generate a seminorm, not a norm.
 
\begin{lem}\label{lem:P-modular}
Let $\phi \in \Phiw(\Omega)$.
Then $\rhoBV$ satisfies the following properties:
\begin{itemize}
\item[(a)] $\lambda \mapsto \rhoBV(\lambda u)$ is nondecreasing for every $u \in L^1(\Omega)$;
\item[(b)] $\rhoBV(0)=0$;
\item[(c)] $\rhoBV(-u)= \rhoBV(u)$;
\item[(d)] there exists a constant $\beta>0$ such that 
\[
\rhoBV(\beta (\theta u + (1-\theta)v))\le \theta \rhoBV(u) + (1-\theta)\rhoBV(v)
\]
for every $u, v \in L^1(\Omega)$ and for every $\theta \in (0, 1)$;
\item[(e1)]$\rhoBV(\lambda u)=0$ for all $\lambda>0$ implies $u=0$.
\end{itemize}
If additionally $\phi$ satisfies \adec{}, then 
\begin{itemize}
\item[(e2)] $\rhoBV(u)=0$ if and only only if $u=0$.
\end{itemize}
Moreover, if $\phi$ is convex or left-continuous, then so is $\rhoBV$.
\end{lem}

\begin{proof}
We consider only $\rholiminf$ for the 
first four properties, since these properties are known for $\rho_\phi$ and thus follow for the 
sum $\rho_\phi + \rholiminf$. 

(a) Let $0\le \lambda_1 \le \lambda_2$. It is enough to show that $\rholiminf(\lambda_1 u) \le \rholiminf(\lambda_2 u)$. 
Let $(v_i)$ be a sequence that gives $\rholiminf(\lambda_2 u)$. Then $\frac{\lambda_1}{\lambda_2}v_i \to \lambda_1 u$, and thus
\[
\rholiminf(\lambda_1 u) \le \lim_{i \to \infty} \rho_\phi(v_i) = \rholiminf(\lambda_2 u).
\]

(b) If $u=0$ a.e., then it can be approximated by the constant sequence $(u)$, and hence $\rholiminf(u)=0$. 

(c) This follows directly from the definition of $\rholiminf$.

(d) Note that by \cite[Corollary 2.2.2]{HarH_book} there exists $\beta >0$ such that $\phi(\beta (\theta u + (1-\theta)v))\le \theta \phi(u) + (1-\theta)\phi(v)$.
Let $u, v \in L^1(\Omega)$ and $\theta \in (0, 1)$. 
Choose $(u_i)$ and $(v_i)$ which give $\rholiminf(u)$ and $\rholiminf( v)$. Then $\theta u_i + (1-\theta) v_i \to \theta u + (1-\theta) v$ in $L^1(\Omega)$. Hence we obtain
\[
\begin{split}
\rholiminf ( \beta(\theta u + (1-\theta) v))
&\le \liminf_{i \to \infty} \rho_\phi(\beta |\nabla (\theta u_i + (1-\theta) v_i)|)\\
&\le \liminf_{i \to \infty} \int_\Omega \phi(x, \beta\theta |\nabla u_i| + \beta(1-\theta) |\nabla v_i|)\, dx\\
&\le \liminf_{i \to \infty} \int_\Omega \theta \phi(x, |\nabla u_i| )+ (1-\theta) \phi(x, |\nabla v_i|)\, dx\\
&= \theta \lim_{i \to \infty}\rho_\phi(|\nabla u_i|) + (1-\theta) \lim_{i \to \infty} \rho_\phi(|\nabla v_i|)\\
&= \theta \rholiminf(u) + (1-\theta) \rholiminf( v),
\end{split}
\] 
where in the third inequality we used the quasiconvexity of $\phi$.

(e1) If $ \rhoBV(\lambda u)=0$ for all $\lambda >0$, then $\rho_\phi(\lambda u)=0$ for all $\lambda >0$.
Since $\rho_\phi$ is a semimodular, it follows that $u=0$ a.e., see \cite[Lemma~3.2.2]{HarH_book}.

(e2)
Assume that \adec{} holds. If $ \rhoBV(u)=0$, then $\rho_\phi(u)=0$. 
Thus $\phi(x,|u(x)|)=0$ a.e., and \adec{} implies that $u=0$ a.e.

If $\phi$ is convex, then we may choose $\beta =1$ in (d), and hence $\rhoBV$ is convex. 

Assume that $\phi$ is left-continuous.
If $\lambda \in(0, 1)$, then $\rholiminf(\lambda u) \le  \rholiminf(u)$ by property (a). 
We next consider the opposite inequality.
Let $(\lambda _i)$ be a sequence converging to $1$ from below. For every $i$ we choose $u_i \in \Lspace{\Omega}$ such that
$ \rho_\phi(|\nabla u_i|) \le \rholiminf(\lambda_i u) + \frac1i$ and $\|\lambda_i u - u_i \|_1 < \frac1i$. Further, $\lambda_i u \to u$ in $L^1(\Omega)$. 
Let $\epsilon >0$ and choose $i_0>\frac2\epsilon$ such that $\|u- \lambda_i u\|_1 < \frac{\epsilon}2$ for all $i \ge i_0$. Then
\[
\|u- u_i\|_1\le \|u-\lambda_i u\|_1 + \|\lambda_i u- u_i\|_1 <\epsilon
\]
for all $i \ge i_0$, and thus $u_i \to u$ in $L^1(\Omega)$. We obtain
\begin{align*}
\rholiminf(u) 
&\le 
\liminf_{i \to \infty} \rho_\phi(|\nabla u_i|)
\le 
\liminf_{i \to \infty}\big( \rholiminf(\lambda_i u) + \tfrac1i \big) 
= 
\liminf_{i \to \infty} \rholiminf(\lambda_i u).
\end{align*}
Thus $\rholiminf(u) = \lim_{i \to \infty} \rholiminf(\lambda_i u)$.
\end{proof}

\begin{rem}
Note that our conditions differ what Musielak \cite[Definition 1.1, p.~1]{Mus83}
calls a pseudomodular. He requires (b), (c)  and
\[
\rho(\theta u + (1-\theta)v)\le  \rho(u) + \rho(v)
\]
for every $u, v$ and for every $\theta \in (0, 1)$.  
However, our condition (d) is more useful when dealing with the Luxemburg norm. 
\end{rem}

We prove the lower semi-continuity of $\rholiminf$; note that we can handle $\Phiw$, since 
we assume strong rather than weak convergence, in contrast to Lemma~\ref{lem:lsc}

\begin{lem}[Lower semicontinuity of the modular]\label{lem:lsc-rholiminf}
Let $\phi \in \Phiw(\Omega)$, $u_i \in \Lspace{\Omega}$ for $i=1, 2, \ldots$, $u \in L^1(\Omega)$, and
$u_i \to u$ in $L^{1}(\Omega)$. Then 
\[
\rholiminf(u) \le \liminf_{i\to \infty} \rholiminf(u_i). 
\]
Moreover, if the limit inferior is finite, then $u \in \Lbar{\Omega}$.
\end{lem}
\begin{proof}
 By the definition of $\rholiminf(u_i)$ 
we can find $u_i'\in \Lspace{\Omega}$ with 
\[
\rho_\phi(|\nabla u_i'|) \le \rholiminf(u_i) + \tfrac1i
\qquad\text{and}\qquad
\|u_i-u_i'\|_1 \le \tfrac1i.
\]
If follows that $u_i'\to u$ in $L^1(\Omega)$. By the definition of $\rholiminf(u)$ 
we conclude that 
\[
\rholiminf(u) 
\le 
\liminf_{i\to \infty} \rho_\phi(|\nabla u_i'|)
\le
\liminf_{i\to \infty} (\rholiminf(u_i) + \tfrac1i)
= 
\liminf_{i\to \infty} \rholiminf(u_i). \qedhere
\]
\end{proof}

We define the Luxemburg (quasi-)norm in the usually way:
\[
\|u\|_\rho := \inf \Big\{\lambda >0 :\rho\big(\tfrac{u}{\lambda}\big) \le 1 \Big\}.
\]
If $\phi$ is a convex semimodular then 
$u \mapsto \|u\|_\rho$ is a norm by \cite[Theorem 1.5, p.~2]{Mus83}. As in \cite[Lemma~3.2.2]{HarH19}, 
one can check that (a)--(d) from Lemma~\ref{lem:P-modular} imply that $u \mapsto \|u\|_\rho$ is a 
quasi-seminorm.

We next compare $\|\cdot \|_{\rholiminf}$ 
induced by the modular via 
the preceding formula with $\|\cdot\|\liminftag$ from Definition~\ref{defn:norms}.

\begin{prop}
If $\phi \in \Phiw(\Omega)$, then $\|\cdot\|\liminftag = \|\cdot\|_{\rholiminf}$.
\end{prop}
\begin{proof}
Let us first show that $\|\cdot\|\liminftag  \ge \|\cdot\|_{\rholiminf}$. If the left-hand side equals $\infty$ there is nothing to prove, and thus we may
assume without loss of generality that $\|u\|\liminftag < \infty$. Choose $u_i\in \Lspace{\Omega}$
with $u_i \to u$ in $L^1(\Omega)$ and $\|\nabla u_i\|_\phi \to \|u\|\liminftag$. 
Define $\lambda_i:=\|\nabla u_i\|_\phi+\epsilon$ and $\lambda:=\|u\|\liminftag+\epsilon$ for $\epsilon>0$.
Then $\|\nabla \frac{u_i}{\lambda_i}\|_\phi < 1$ and so $\rho_\phi(|\nabla \frac{u_i}{\lambda_i}|)\le 1$. 
Furthermore, 
\[
\big\|\tfrac u\lambda-\tfrac{u_i}{\lambda_i}\big\|_1
= 
\big\|\tfrac{u-u_i}{\lambda_i} + (\tfrac1\lambda-\tfrac{1}{\lambda_i})u\big\|_1
\le 
\tfrac{1}{\lambda_i}\|u-u_i\|_1 + (\tfrac1\lambda-\tfrac{1}{\lambda_i}) \|u\|_1 \to 0
\]
so that $\frac{u_i}{\lambda_i} \to \frac u\lambda$ in $L^1(\Omega)$. 
Since $(\frac{u_i}{\lambda_i})$ is a valid test sequence 
for $\rholiminf$, we conclude that 
\[
\rholiminf(\tfrac u\lambda) 
\le 
\liminf_{i \to \infty} \rho_\phi\big(|\nabla \tfrac{u_i}{\lambda_i}|\big) 
\le 
1. 
\]
Hence $\|u\|_{\rholiminf} \le \lambda = \|u\|\liminftag + \epsilon$; the inequality follows as $\epsilon\to 0^+$. 

To prove $\|\cdot\|\liminftag  \le \|\cdot\|_{\rholiminf}$, we may assume 
that $\|u\|_{\rholiminf}<\infty$. Since $\rholiminf(\frac u{\|u\|_{\rholiminf}+\epsilon})< 1$, 
we can choose $u_i\in \Lspace{\Omega}$ with $\rho_\phi(|\nabla u_i|)\le 1$ and 
$u_i\to \frac u{\|u\|_{\rholiminf}+\epsilon}$ in $L^1$. 
Thus 
\[
\big\|\tfrac u{\|u\|_{\rholiminf}+\epsilon}\big\|\liminftag \le \liminf_{i \to \infty} \|\nabla u_i\|_\phi \le 1 
\]
and so $\|u\|\liminftag \le \|u\|_{\rholiminf}+\epsilon $. The claim again follows as $\epsilon\to 0^+$.
\end{proof}



\section{Functionals with fixed boundary values}
\label{sect:boundaryValues}

Let $\phi \in \Phiw(\Omega)$. 
We say that $u \in L^{1, \phi}(\Omega)$ belongs to 
 $L^{1, \phi}_{0}(\Omega)$ if there exists a $C_0^\infty(\Omega)$-sequence $(\xi_i)$ such that $\xi_i \to u$ in 
 $L^{1, \phi} (\Omega)$. Note that $L^{1, \phi}_{0}(\Omega) \subset W^{1, 1}_0(\Omega)$.


\begin{lem}\label{lem:reflexive}
Let $\phi \in \Phiw(\Omega)$ satisfy \ainc{} and \adec{}. Then $\Lspace{\Omega}$ and $L^{1, \phi}_{0}(\Omega)$ are reflexive.
\end{lem}

\begin{proof}
Let $(u_i)\subset L^{1, \phi}(\Omega)$ be a bounded sequence.
Since $L^{1, \phi}(\Omega) \subset \BV(\Omega)$, a subsequence, denoted again $(u_i)$, converges in $L^1(\Omega)$ to some function $u\in \BV(\Omega)$. 
On the other hand, $(\partial_{x_k} u_i)$ is a bounded sequence in $L^\phi(\Omega)$, which 
is a reflexive space \cite[Theorem 3.6.6]{HarH_book}. 
Hence we can find $v_1, \ldots, v_n \in L^{\phi}(\Omega)$ such that
$\partial_{x_k} u_i \rightharpoonup v_k$ in $L^\phi(\Omega)$, up to a subsequence.
So we need only prove that 
$\nabla u = (v_1, \ldots, v_n)$. Let $k\in\{1, \ldots, n\}$ and let $h \in C^\infty_0(\Omega)$ be a test function. We obtain by the previous convergences that 
\[
\int_\Omega u  \,\partial_{x_k} h \, dx  = \lim_{i \to \infty} \int_\Omega u_i \,\partial_{x_k} h \, dx 
= -  \lim_{i \to \infty} \int_\Omega h \,\partial_{x_k} u_i \, dx = -\int_\Omega h v_k \, dx,
\]
and thus $v_k = \partial_{x_k} u$.

The usual diagonal argument shows that $L^{1, \phi}_{0}(\Omega)$ is a closed subspace of $\Lspace{\Omega}$. 
Hence it is also reflexive.
\end{proof}

We will study $\phi^p$-energies when $p\ge 1$ is a constant. For that purpose we note some 
properties of $\phi^p$ when $\phi\in\Phiw(\Omega)$. 

\begin{lem}\label{lem:phi^p}
Let $p\ge 1$ and $\phi \in \Phiw(\Omega)$. Then $\phi^p\in\Phiw(\Omega)$ satisfies \ainc{p} 
and, moreover,
\begin{enumerate}
\item if $\phi$ satisfies \adec{q}, then $\phi^p$ satisfies \adec{qp};
\item if $\phi$ is left-continuous, then so is $\phi^p$;
\item if $\phi$ satisfies \azero{}, then so does $\phi^p$.
\end{enumerate}
\end{lem}

\begin{proof}
Since $t\mapsto t^p$ is continuous and increasing in $[0, \infty)$, 
the first three requirements of Definition~\ref{def2-1} hold for $\phi^p$.
Since $\phi$ satisfies \ainc{1} and $t \mapsto t^p$ is increasing, we obtain by $\frac{\phi^p(x, t)}{t^p} = \big( \frac{\phi(x, t)}{t}\big)^p$ that $\phi^p$ satisfies \ainc{p}. 

(1) Assume that $\phi$ satisfies \adec{q}. Then by $\frac{\phi^p(x, t)}{t^{pq}} = \big( \frac{\phi(x, t)}{t^q}\big)^p$ we see that $\phi^p$ satisfies \adec{pq}. 

(2) This claims follows since $t \mapsto t^p$ is continuous.

(3) It follows from \azero{} that $\phi(x, \beta)^p \le 1 \le \phi(x, \frac1\beta)^p$, 
which is \azero{} of $\phi^p$. 
%
\end{proof}


\subsection{\texorpdfstring{$\Gamma$}{Gamma}-convergence}~\\
Assume that $u_0 \in L^{1, \phi^r}(\Omega) \cap L^1(\Omega)$ for some $r >1$.
For $p \in (1, r)$ we consider $E_p: L^1(\Omega) \to [0, \infty]$ and 
$E: L^1 (\Omega) \to [0, \infty]$ from the introduction, see \eqref{eq:Ep}. 
Note that $E$ is a variant of $\rholiminf$ where we have fixed the boundary values. 


As we mentioned in the introduction, we want to show 
that $E_{p_i}$ $\Gamma$-converges to $E$ in the topology of $L^1(\Omega)$ as $p_i \to 1^+$. 
We do so by first considering the 
$\Gamma$-liminf inequality, and then finding the recovery sequence in Lemma~\ref{lem:Gamma-limsup}. 
This establishes the two properties of 
$\Gamma$-convergence and completes the proof of Theorem~\ref{thm:Gamma}(1).

\begin{lem}\label{lem:Gamma-liminf}
Let $\phi \in \Phiw(\Omega)$.
Assume that  $u_0 \in L^{1, \phi^r}(\Omega)$ for some $r >1$.
Let $u \in L^1(\Omega)$, $u_i\to u$ in $L^1(\Omega)$ and $p_i \to 1^+$.
Then
\[
 E(u) \le \liminf_{i \to \infty} E_{p_i}(u_i).
\] 
\end{lem}

\begin{proof}
If $K:=\liminf_{i \to \infty} E_{p_i} (u_i) =\infty$, then there is nothing to prove, 
so we assume that $K<\infty$. 
We restrict our attention to a subsequence with $\lim_{i \to \infty} E_{p_i} (u_i)=K$ 
and $u_i- u_0\in L^{1,\phi^{p_i}}_{0}(\Omega)$. 
Since $u_i \to u$ in $L^1(\Omega)$, we obtain 
by the definition of $ E(u)$ that
\[
 E(u) \le \liminf_{i \to \infty} \int_\Omega \phi(x, {|\nabla u_i|) \, dx}.
\]
From Young's inequality we conclude that $ab \le a^p + (p-1)(\frac bp)^{p'}$. 
Using this with $a=\phi(x, |\nabla u_i|)$, $b=1$ and $p=p_i$, we continue the previous estimate by 
\begin{equation*}
\begin{split}
E(u) 
&\le \liminf_{i\to \infty} \Big(\int_\Omega \phi^{p_i}(x, {|\nabla u_i|}) \, dx + |\Omega|(p_i-1) p_i^{-p_i'} \Big)
= \liminf_{i\to \infty} \int_\Omega \phi^{p_i}(x, {|\nabla u_i|}) \, dx,
\end{split}
\end{equation*}
where the equality holds since $(p_i-1) p_i^{-p_i'} \to 0 \cdot \frac1e=0$ as $p_i \to 1^+$.
\end{proof}

We now turn our attention to the $\Gamma$-limsup property. For this part we need to assume \dec{}.
Recall that \adec{} implies \dec{} if $\phi$ is convex \cite[Lemma~2.2.6]{HarH19}. 

\begin{lem}\label{lem:Gamma-limsup}
Assume that $\phi \in \Phiw(\Omega)$ satisfies \azero{} and \dec{}.
Let $u\in L^1(\Omega)$ and $u_0\in L^{1, \phi^r}(\Omega)$ for some $r>1$.
For every $p_i\to 1^+$ there exist $u_i \in \Lspace{\Omega}$ such that $u_i -u_0 \in L^{1, \phi}_0(\Omega)$,
$u_i \to u$ in $L^1(\Omega)$ and 
\[
 E(u) \ge \limsup_{i \to \infty} E_{p_i}(u_i).
\]
\end{lem}

\begin{proof}
If $ E(u) =\infty$, then any approximating sequence in $u_0+C_0^\infty(\Omega)$ works \cite[Corollary~3.7.10]{HarH_book}, 
so we assume that $ E(u) < \infty$. 
By the definition of $E(u),$ there exists a sequence $(v_k)$ from $u_0+L^{1, \phi}_0(\Omega)$ such that
$v_k \to u$ in $L^1(\Omega)$ and 
\[
 E(u) = \lim_{k \to \infty} \int_\Omega \phi(x, |\nabla v_k|) \, dx.
\]
Since $v_k -u_0\in L^{1, \phi}_{0}(\Omega)$, there exists a sequence $(\zeta_j^k)$ from 
$C^\infty_0(\Omega)$ with $u_0 + \zeta_j^k \to v_k$ in $\Lspace{\Omega}$.
Since $\phi$ satisfies \dec{} this yields by Lemma 3.1.6 of \cite{HarH_book} that
\[
\int_\Omega \phi(x, {|\nabla v_k|}) \, dx = \lim_{j \to \infty} \int_\Omega \phi(x, {|\nabla (u_0 + \zeta_j^k)|}) \, dx.
\]
Thus by the usual diagonal argument, we can choose $\xi_j := \zeta_{k_j}^j$ such that 
$u_0 + \xi_j \to u$ in $L^1(\Omega)$ and 
\[
 E(u)= \lim_{j \to \infty} \int_\Omega \phi(x, {|\nabla (u_0 + \xi_j)|}) \, dx. 
\]

By Hölder's inequality with exponent $q>1$ we obtain
\[
\int_\Omega \phi(x, {|\nabla (u_0 + \xi_j)|})^{\frac1q + \frac r{q'}} \, dx
\le \Big( \int_\Omega \phi(x, {|\nabla (u_0 + \xi_j)|}) \, dx \Big)^{\frac1{q}} 
\Big(\int_\Omega \phi(x, {|\nabla (u_0 + \xi_j)|})^{r} \, dx \Big)^{\frac1{q'}}.
\]
The second integral is finite since $u_0 +\xi_j\in W^{\phi^r}(\Omega)$. 
Thus we may choose $\overline{q_j}\in (1, 1+\frac1j)$ close enough to $1$ that the second factor on the right hand side is at most 
$1+ \frac1j$ for all $q\in [1, \overline{q_j}]$. Since $p_i\to 1^+$ there exists $i_j\ge j$ such that 
$q_i:=\frac{r-1}{r-p_i}\le \overline{q_j}$ for all $i\ge i_j$. 
Note that $\frac1{q_i} + \frac r{q_i'} = p_i$. We now choose 
$u_i := u_0 +\xi_j$ for all $i \in [\max\{i_1, \ldots, i_{j}\}, i_{j+1})$. 
Then 
\[
\int_\Omega \phi(x, |\nabla u_i|)^{p_i} \, dx 
\le 
\Big(1+\frac1{j(i)}\Big) \Big( \int_\Omega \phi(x, |\nabla u_i|) \, dx \Big)^{\frac1{q_i}},
\]
where $j(i)\to \infty$ as $i\to \infty$ since $\overline {q_j}>1$ so only finitely many values of 
$i$ can have $q_i>\overline {q_j}$.

Next we combine all the above estimates, and obtain
\[
\begin{split}
E(u) 
= \lim_{i \to \infty} \int_\Omega \phi(x, | \nabla u_i|) \, dx 
&\ge \limsup_{i \to \infty} \Big(1+\frac1{j(i)}\Big)^{-q_i} \Big(\int_\Omega \phi(x, |\nabla u_i|)^{p_i} \, dx\Big)^{q_i}\\
&= \limsup_{i \to \infty} \int_\Omega \phi(x, {|\nabla u_i|})^{p_i} \, dx = \limsup_{i \to \infty} E_{p_i}(u_i),
\end{split}
\]
where the last equality follows since $u_i- u_0 = \xi_j \in L^{1, \phi}_{0}(\Omega)$.
\end{proof}

\subsection{A sequence of minimizers}~\\
We continue by establishing some results about sequences of minimizers of the functionals 
$E_p$ from \eqref{eq:Ep}. These results are in the spirit of Juutinen \cite{Juu05}.

\begin{lem}\label{lem:min-exists}
Let $\phi \in \Phi_c(\Omega)$ satisfy \azero{} and \adec{}, 
$p>1$ and $u_0 \in L^{1, \phi^p}(\Omega)$. Then there exists a unique function $u_p \in L^{1, \phi^p}(\Omega)$ which satisfies $u_p- u_0 \in L^{1, \phi^p}_0(\Omega)$ and 
\[
E_p(u_p) = \inf_{u \in L^1(\Omega)} E_p(u).
\]
\end{lem}

\begin{proof}
Let $(u_j)$ be a minimizing sequence of $E_p$. Since
\[
\inf_{u \in L^1(\Omega)} E_p(u) \le E_p(u_0) = \int_\Omega \phi(x, |\nabla u_0|)^p \, dx < \infty,
\]
we may assume that $(|\nabla u_j|)$ is bounded in $L^{\phi^p}(\Omega)$. 
By the Poincar\'e inequality  in $W^{1, 1}_0 (\Omega)$, we obtain that 
$(u_j-u_0)$ is bounded in $L^1(\Omega)$:
\[
\|u_j - u_0\|_1
\lesssim \|\nabla (u_j -u_0)\|_1 
\lesssim \|\nabla u_j\|_{\phi^p}+ \|\nabla u_0\|_{\phi^p},
\]
where the embedding in the last inequality holds since  $\phi^p$ satisfies \azero{} by Lemma~\ref{lem:phi^p}.
Since $\phi^p$ satisfies \ainc{} and \adec{}, 
$L^{1, \phi^p}_{0}(\Omega)$ is reflexive, by Lemma~\ref{lem:reflexive}. 
Thus there exists a subsequence $(u_{j_k}-u_0)$ and $u \in L^{1, \phi}_0(\Omega)$ such that $\nabla (u_{j_k} - u_0) \rightharpoonup \nabla u$ in $L^{\phi^p}(\Omega; \Rn)$. Set 
$u_p:=u+u_0$.Since $\phi\in \Phic(\Omega)$, Lemma~\ref{lem:lsc} implies that
\[
E_p(u_p) \le \liminf_{k\to \infty} E_p(u_{j_k}).
\]
Thus $u_p$ is a minimizer of $E_p$. 

Since $\phi$ is convex and $t\mapsto t^p$ strictly convex, we obtain that $t \mapsto \phi(x, t)^p$ is strictly convex. The usual argument yields uniqueness, namely, 
if $u$ and $v$ are distinct minimizers, then we obtain a contradiction from 
$E_p(\frac{u+v}2)< \frac12 (E_p(u)+ E_p(v))$. 
\end{proof}

Take a sequence $(p_j)$ such that $p_j\to 1^+$ and a sequence of minimizers $(u_j)$, where $u_j$ is the minimizer of $E_{p_j}$. Let us also assume that the $\phi^{p_j}$-modular of $u_0$ is finite for every $j$. Then the limit function exists in $\Lbar\Omega$. 

\begin{lem}\label{lem:subseq-BV}
Let  $p_j\to 1^+$ and $\phi \in \Phiw(\Omega)$. Assume that $\rho_{\phi^{p_j}}(u_0) \le K$ for every $j$. 
Let $(u_j)$ be a sequence of minimizers of $E_{p_j}$ with boundary value $u_0$.
Then there exist $u \in \Lbar{\Omega}$ and a subsequence $(u_{j_i})$ converging to $u$ in $L^1(\Omega)$ with
\[
E(u) 
\le
\liminf_{i\to \infty} E_{p_{j_i}}(u_{j_i}). 
\]
\end{lem}
\begin{proof}
Since $u_j$ is a minimizer, we see as in the previous lemma that $E_{p_j}(u_j)\le K$. 
By $L^{\phi^{p_1}}(\Omega) \hookrightarrow L^\phi(\Omega)$ and 
$\|u\|_\phi\le \rho_\phi(u)+1$ we conclude that $\|\nabla u_j\|_\phi$ is bounded.
Since $u_j -u_0\in W^{1,1}_0(\Omega)$, we obtain 
as in the previous lemma by the Poincar\'e inequality that
$\|u_j-u_0\|_1 \lesssim 1 + \|\nabla u_0\|_{\phi^{p_1}}$.
Thus $(u_j-u_0)$ is bounded in $W^{1,1}_0(\Omega)$.
Since $W^{1,1}_0(\Omega)\hookrightarrow \hookrightarrow L^1(\Omega)$, 
there exists a subsequence $(u_{j_i})$ converging to $u$ in $L^1(\Omega)$.
The claim follows from the first property of $\Gamma$-convergence by Theorem~\ref{thm:Gamma}(1). 
Finally, Lemma~\ref{lem:lsc-rholiminf} yields that $u \in \Lbar {\Omega}$.
\end{proof}

\begin{prop}\label{prop:minimizer-E}
Let $\phi \in \Phiw(\Rn)$ and $\rho_{\phi^{p_j}}(u_0) \le K$ for every $j$.
The function $u\in \Lbar{\Omega}$ from Lemma~\ref{lem:subseq-BV} minimizes the $E$-energy, i.e.
\[
E(u) 
=
\inf_{v\in L^1(\Omega)} E(v).
\]
\end{prop}

\begin{proof}
Assume by contradiction that there exists $v \in L^1(\Omega)$ such that 
$E(v)< E(u)$. 
Let $p_i\to 1^+$. 
By the second property of $\Gamma$-convergence from Theorem~\ref{thm:Gamma}(1), 
there exists a sequence $(v_i)$ such that 
$v_i - u_0 \in L^{1, \phi}_0(\Omega)$, 
$v_i \to v$ in $L^1(\Omega)$ and 
\[
E( v) \ge \limsup_{i \to \infty} E_{p_i}(v_i).
\]
By the assumption and Lemma~\ref{lem:subseq-BV}, we have
\[
\limsup_{i \to \infty} E_{p_i}(v_i) \le E(v)< E(u)\le \liminf_{j \to \infty} E_{p_{i_j}}(u_{i_j}).
\] 
Let $\epsilon := \liminf_{j \to \infty} E_{p_{i_j}}(u_{i_j}) - \limsup_{i \to \infty} E_{p_i}$, and choose $M$ so large that 
\[
E_{p_i}(v_i) < \limsup_{i \to \infty} E_{p_i}(v_i) + \frac{\epsilon}{2}
\quad \text{and} \quad 
E_{p_{i_j}}(v_{i_j}) > \liminf_{j \to \infty} E_{p_{i_j}}(u_{i_j})- \frac{\epsilon}{2}
\] 
as $i, j\ge M$. Thus for all $i, j \ge M$ we have 
\[
E_{p_i}(v_i) < E_{p_{i_j}}(u_{i_j}).
\]
Then we just choose $j\ge M$ with ${i_j} \ge M$ and obtain
\[
E_{p_{i_j}}(v_{i_j}) < E_{p_{i_j}}(u_{i_j}).
\]
But this contradicts that $u_{i_j}$ minimizes the $E_{p_{i_j}}$-energy with the boundary values $u_0$, 
so the assumption $E(v)<E(u)$ was wrong and the claim is proved.
\end{proof}


\section{Functionals with the fidelity term}
\label{sect:fidelityTerm}

Assume then that $f \in L^2(\Omega)$.
For $p >1$ we defined $F_p: L^2(\Omega) \to [0, \infty]$ and 
$F: L^2 (\Omega) \to [0, \infty]$ in the introduction, see \eqref{eq:Fp}. 
Note that $F$ is another a variant of $\rholiminf$, this time including a ``fidelity term'' in $L^2$. 
Since this implies that $u\in L^2(\Omega)$, it seems natural to consider the 
convergence in $L^2(\Omega)$, although convergence in $L^1(\Omega)$ is also a possibility.

\subsection{\texorpdfstring{$\Gamma$}{Gamma}-convergence}~\\
We want to show 
that $F_{p_i}$ $\Gamma$-converges to $F$ in the topology of $L^2(\Omega)$ as $p_i \to 1^+$, i.e.\ 
to prove Theorem~\ref{thm:Gamma}(2). We do so by first considering the 
$\Gamma$-liminf inequality, and then finding the recovery sequence in Lemma~\ref{lem:Gamma-limsup-F}.

\begin{lem}\label{lem:Gamma-liminf-F}
Let $u \in L^2(\Omega)$ and assume that $(u_i)$ is a sequence of $L^2$-functions converging to $u$ in $L^2(\Omega)$ and that $p_i \to 1^+$.
Then
\[
 F(u) \le \liminf_{i \to \infty} F_{p_i}(u_i).
\] 
\end{lem}

\begin{proof}
If $K:=\liminf_{i \to \infty} F_{p_i} (u_i) =\infty$, then there is nothing to prove, 
so we assume that $K<\infty$. 
We restrict our attention to a subsequence with $\lim_{i \to \infty} F_{p_i} (u_i)=K$ 
and $u_i\in L^{1,\phi^{p_i}}(\Omega)\cap L^2(\Omega)$. 
Since $u_i \to u$ in $L^2(\Omega)$, we obtain 
by the definition of $ F(u)$ that
\[
 F(u) \le \liminf_{i \to \infty} \Big(\int_\Omega \phi(x, |\nabla u_i|) +  |u_i-f|^2 \, dx \Big).
\]
From Young's inequality we conclude that $ab \le a^p + (p-1)(\frac bp)^{p'}$. 
Using this with $a=\phi(x, |\nabla u_i|)$, $b=1$ and $p=p_i$, we continue the previous estimate by 
\begin{equation*}
\begin{split}
F(u) 
&\le \liminf_{i\to \infty} \Big(\int_\Omega \phi^{p_i}(x, {|\nabla u_i|}) \, dx + |\Omega|(p_i-1) p_i^{-p_i'} +  \int_\Omega  |u_i-f|^2 \, dx\Big)\\
&= \liminf_{i\to \infty} \Big(\int_\Omega \phi^{p_i}(x, {|\nabla u_i|}) \, dx+  \int_\Omega  |u_i-f|^2 \, dx\Big),
\end{split}
\end{equation*}
where the equality holds since $(p_i-1) p_i^{-p_i'} \to 0 \cdot \frac1e=0$ as $p_i \to 1^+$.
\end{proof}

We now turn our attention to the $\Gamma$-limsup property. 
Note that in contrast to Lemma~\ref{lem:Gamma-limsup} we now have an assumption regarding density 
of smooth functions. The difference comes from the fact that $L^{1,\phi}_0$ is defined as the completion 
of $C^\infty_0$-functions, so smooth functions are automatically dense in that setting. 
We refer to \cite[Corollary 4.4]{Juu_pp} for a sufficient condition for the density of 
$C^\infty(\overline\Omega)$ in $W^{1,\phi}(\Omega)$ where $\Omega$ is an $(\epsilon,\infty)$-domain.

\begin{lem}\label{lem:Gamma-limsup-F}
Assume that $\phi \in \Phiw(\Omega)$ satisfies \azero{} and \dec{}, and that $C^\infty(\overline\Omega)$ is 
dense in $L^{1,\phi}(\Omega)\cap L^2(\Omega)$.
For every $u\in L^2(\Omega)$ and $p_i\to 1^+$ there exist $u_i \in L^{1, \phi}(\Omega) \cap L^2(\Omega)$ such that
$u_i \to u$ in $L^2(\Omega)$ and 
\[
F(u)  \ge \limsup_{i \to \infty} F_{p_i}(u_i).
\]
\end{lem}

\begin{proof}
If $ F(u) =\infty$, any approximating sequence in $C_0^\infty(\Omega)$ works \cite[Corollary~3.7.10]{HarH_book}, so we assume that $ F(u) < \infty$.  Thus  there exists a sequence $(v_i)$ from $L^{1, \phi}(\Omega)$ such that
$v_i \to u$ in $L^{2}(\Omega)$ and 
\[
F(u) = \lim_{i \to \infty} \int_\Omega \phi(x, |\nabla v_i|) + |v_i-f|^2\, dx .
\]
Using the density assumption, we can find $\xi_j^i\in C^\infty(\overline\Omega)$
with $\xi_j^i\to v_i$ in $L^{1,\phi}(\Omega)\cap L^2(\Omega)$. Then by a diagonal argument and \cite[Lemma 3.1.6]{HarH_book} 
we find $u_i:=\xi_{j_i}^i\in C^\infty(\overline\Omega)$ with 
$u_i\to u$ in $L^2(\Omega)$ and 
\[
F(u) = \lim_{i \to \infty} \int_\Omega \phi(x, |\nabla u_i|) + |u_i-f|^2\, dx .
\]
The proof is concluded in the same way as Lemma~\ref{lem:Gamma-limsup}. 
\end{proof}

\subsection{A sequence of minimizers}

\begin{lem}\label{lem:min-exists-F}
Let $\phi \in \Phic(\Omega)$ satisfy \adec{}, $p>1$ and $f \in L^{2}(\Omega)$. Then there exists a unique function $u_p \in L^{1, \phi^p}(\Omega) \cap L^2(\Omega)$ which satisfies
 \[
F_p(u_p) = \inf_{u \in L^2(\Omega)} F_p(u) .
\]
\end{lem}

\begin{proof}
Let $(u_j)$ be a minimizing sequence of $F_p$. Since
\[
\inf_{u \in L^{1, \phi^p}(\Omega) \cap L^2(\Omega)}  F_p(u) 
\le 
F_p(0) 
= 
\int_\Omega \phi(x, 0)^p  + |0-f|^2 \, dx < \infty,
\]
we may assume that $(|\nabla u_j|)$ is bounded in $L^{\phi^p}(\Omega)$ and $(u_j)$ is bounded in $L^2(\Omega)$. 
Note that $\phi^p$ satisfies \ainc{} and \adec{} by Lemma~\ref{lem:phi^p}.
Thus $L^2(\Omega)$ and $\Lspace{\Omega}$ are reflexive by Lemma~\ref{lem:reflexive}, and hence there exists a subsequence $(u_{j_k})$ and $u \in L^{1, \phi^p}(\Omega) \cap L^2(\Omega)$ such that $\nabla u_{j_k} \rightharpoonup \nabla u$ in $L^{\phi^p}(\Omega; \Rn)$ and $ u_{j_k} \rightharpoonup u$ in $L^2(\Omega)$. 
Thus by weak lower semi-continuity (Lemma~\ref{lem:lsc}) we have 
\[
F_p(u) \le \liminf_{k\to \infty} F_p(u_{j_k}).
\]
Thus $u$ is a minimizer. Uniqueness is proved as in Lemma~\ref{lem:min-exists}.
\end{proof}

Let us take a sequence $(p_j)$ such that $p_j\to 1^+$ and a sequence of minimizers $(u_j)$, where $u_j$ is the minimizer of $F_{p_j}$. Then the limit function exists in $L^{1,\phi}(\Omega) \cap L^2(\Omega)$.

\begin{lem}\label{lem:subseq-BV-F}
Let $\phi \in \Phic(\Omega)$. 
Assume that  $f \in L^2(\Omega)$. Let $p_j\to 1^+$ and let $(u_j)$ be a sequence of minimizers of $F_{p_j}$.
Then there exist $u \in \Lbar\Omega \cap L^2(\Omega)$ and a subsequence $(u_{j_i})$ converging to $u$ in $L^2(\Omega)$ with
\[
F(u) 
\le
\liminf_{i\to \infty} F_{p_{j_i}}(u_{j_i}). 
\]
\end{lem}

\begin{proof}
We take a subsequence that gives $\liminf_{i\to \infty} F_{p_{j_i}}(u_{j_i})$, and denote this subsequence by $(u_j)$.
Since $u_j$ is a minimizer of $F_{p_j}$, 
\[
F_{p_j}(u_j) \le \int_\Omega \phi(x, 0)^{p_j}  + |0-f|^2\, dx = \int_\Omega |f|^2\, dx < \infty.
\]
Thus by $L^{\phi^p}(\Omega) \hookrightarrow L^\phi(\Omega)$ we have that $\|\nabla u_j\|_\phi+\|u_j-f\|_2$ is bounded, and by $f\in L^2(\Omega)$ we obtain that $\|u_j\|_2$ is bounded.
Since $L^2(\Omega)$ is reflexive, there exists a subsequence, denoted again by $(u_j)$, and $u \in L^2(\Omega)$ such that $u_j \rightharpoonup u$ in $L^2(\Omega)$.

By the Banach--Saks Theorem, $\frac1k \sum_{j=1}^k u_j \to u$ in $L^2(\Omega)$. 
We then use the definition of $F$ and the convexity of $\phi$ and $t \mapsto t^2$ and obtain that 
\[
\begin{split}
F(u) &\le \liminf_{k \to \infty} F \bigg( \frac1k \sum_{j=1}^k u_j\bigg)
= \liminf_{k \to \infty} \int_\Omega \phi \bigg( x, \frac1k \sum_{j=1}^k |\nabla u_j|\bigg) + \bigg|  \frac1k \sum_{j=1}^k u_j -f\bigg |^2 \, dx \\
&\le \liminf_{k \to \infty} \frac1k \sum_{j=1}^k \int_\Omega \phi ( x, |\nabla u_j|) + |u_j -f|^2 \, dx
\end{split}
\]
Arguing as in Lemma~\ref{lem:Gamma-liminf-F}, we continue: 
\[
\begin{split}
F(u) 
&\le 
\liminf_{k \to \infty} \frac1k \sum_{j=1}^k \int_\Omega \phi ( x, |\nabla u_j|)^{p_j} + |u_j -f|^2  + (p_j-1) p_j^{-p'_j}\, dx
= \liminf_{k \to \infty} \frac1k \sum_{j=1}^k F_{p_j} ( u_j) .
\end{split}
\]
It is well know that if a sequence $(z_j)$ converges to $z$, so does $(\frac1k \sum_{j=1}^k z_j)$. 
Thus we obtain $F(u) \le \lim_{j \to \infty}F_{p_j} ( u_j)$.
Then Lemma~\ref{lem:lsc-rholiminf} yields that $u \in \Lbar\Omega$.
\end{proof}

The next result is proved like Proposition~\ref{prop:minimizer-E}, with the lemmas from 
this section replacing their counterparts from the previous section. 

\begin{prop}
Let $\phi \in \Phic(\Omega)$ and $C^\infty(\overline\Omega)$ be 
dense in $L^{1,\phi}(\Omega)\cap L^2(\Omega)$.  Assume that  $f \in L^2(\Omega)$.
The function $u\in L^{1,\phi}(\Omega)\cap L^2(\Omega)$ from Lemma~\ref{lem:subseq-BV-F} minimizes the $F$-energy, i.e.
\[
F(u) 
=
\inf_{v\in L^2(\Omega)} F(v).
\]
\end{prop}
%


\section{Double phase and variable exponent cases}
\label{sect:specialCases}

We considered the $\Gamma$-convergence of the double phase functional 
\[
I_\epsilon (u) := \int_\Omega |\nabla u|^{1+\epsilon} + a(x) |\nabla u|^{2} + |u-f|^2\, dx
\]
in \cite{HarH21}. Specifically, we proved the following, where $\nabla_a u$ is the 
absolutely continuous part of the \BV{}-gradient and $V(u, \Omega)$ is the 
total variation of $u$. (We refer to \cite{AmbFP00} for more information about 
\BV{}-spaces.)

\begin{thm}[Theorem~4.1, \cite{HarH21}]\label{thm:GammaDP}
Suppose that $\Omega$ is an open rectangular cuboid, $a\in C^{0,1}(\overline \Omega)$, 
and assume that $a>0$ $H^{n-1}$-a.e.\ on the boundary $\partial \Omega$. 
Then $I_\epsilon$ $\Gamma$-converges to  $I: \BV(\Omega) \to [0,\infty]$, 
\[
I(u):= V(u, \Omega) + \int_\Omega a(x)|\nabla_a u|^2\, + |u-f|^2 dx,
\]
in $L^1(\Omega)$-topology.
\end{thm}

By the uniqueness of the $\Gamma$-limit, we obtain the following explicit formula for 
$F$:

\begin{cor}\label{cor:GammaDP}
Suppose that $\Omega$ is an open rectangular cuboid, $a \in C^{0,1}(\overline\Omega)$, 
and assume that $a>0$ $\mathcal H^{n-1}$-a.e. on the boundary $\partial \Omega$. 
Then
\[
F(u)= V(u, \Omega) + \int_\Omega a(x)|\nabla_a u|^2\, + |u-f|^2 dx,
\]
when $F$ is defined with the double phase function $\phi(x,t):=t + a(x) t^2$.
\end{cor}
\begin{proof}
Let $u\in L^2(\Omega)$ and $p_i\to 1^+$. 
We use the elementary estimate $(s+t)^p - s^p \ge t^p \ge t - (p-1)$. 
By the second property of $\Gamma$-convergence (Theorem~\ref{thm:Gamma}(2)), 
there exist $u_i \in L^{1, \phi}(\Omega) \cap L^2(\Omega)$ and $p_i\in (1,2)$ with 
$p_i\to 1^+$ such that $u_i \to u$ in $L^2(\Omega)$ and 
\begin{align*}
F(u)
&\ge 
\limsup_{i \to \infty} \int_\Omega (|\nabla u_i| + a(x) |\nabla u_i|^2)^{p_i} + |u_i-f|^2\, dx \\
&\ge 
\limsup_{i \to \infty} \int_\Omega |\nabla u_i|^{p_i} + a(x) |\nabla u_i|^2  + |u_i-f|^2- (p_i-1)\, dx  \\
&\ge 
V(u, \Omega) + \int_\Omega a(x)|\nabla_a u|^2\, + |u-f|^2 dx,
\end{align*}
where the last step is the $\Gamma$-convergence of $I_\epsilon$ (Theorem~\ref{thm:GammaDP}). 

For the opposite inequality we obtain from the second property of $\Gamma$-convergence (Theorem~\ref{thm:GammaDP}) functions $(u_i) \subset \Lspace\Omega$ 
with $u_i\to u$ in $L^1(\Omega)$ and
\begin{align*}
I(u)
&\ge 
\limsup_{i \to \infty} \int_\Omega |\nabla u_i|^{p_i} + a(x) |\nabla u_i|^2 + |u_i-f|^2\, dx \\
&\ge 
\limsup_{i \to \infty} \int_\Omega |\nabla u_i| + a(x)|\nabla u_i|^2 + |u_i-f|^2 - (p_i-1)\, dx  
\ge 
F(u),
\end{align*}
where we used the definition of $F$ in the last step. 
\end{proof}

In the case of bounded functions we were able to prove $\Gamma$-convergence 
under a weaker condition on $a$, namely $\alpha > \frac12$ instead of $\alpha=1$. 

\begin{thm}[Theorem~4.2, \cite{HarH21}]
Suppose that $\Omega$ is a bounded Lipschitz domain, $a\in C^{0,\alpha}(\overline \Omega)$ 
for some $\alpha>\frac12$, and assume that $a>0$ $H^{n-1}$-a.e.\ on the boundary $\partial \Omega$. 
Then $I_\epsilon\big|_{L^\infty}$ $\Gamma$-converges to $I\big|_{L^\infty}$ in $L^1(\Omega)$-topology.
\end{thm}

This 
gives another corollary. The proof is similar to the previous corollary, so we skip it. 

\begin{cor} 
Suppose that $\Omega$ is a bounded Lipschitz domain, $a \in C^{0,\alpha}(\overline\Omega)$ 
for some $\alpha>\frac12$, 
and assume that $a>0$ $\mathcal H^{n-1}$-a.e. on the boundary $\partial \Omega$. 
For bounded $u$, 
\[
F(u)= V(u, \Omega) + \int_\Omega a(x)|\nabla_a u|^2  + |u-f|^2\, dx,
\]
when $F$ is defined with the double phase function $\phi(x,t):=t + a(x) t^2$.
\end{cor}

Next we consider the variable exponent case $\phi(x,t):=t^{p(x)}$, 
where $p:\Omega \to [1,\infty)$. We denote $Y:=\{x \in \Omega: p=1\}$ and define 
$p_\delta(x):=\max\{\delta, p(x)\}$. In \cite[Theorem~1.5]{HarHLT13} 
we proved $\Gamma$-convergence of the $p_\delta$-energy. 
\begin{thm}
Let $\Omega$ be an open rectangular cuboid and let $p$ be strongly log-Hölder continuous, i.e. 
\[
|p(x)-p(z)| \log\Big(e+\frac1{|x-z|}\Big) \le c
\qquad\text{and}\qquad
\lim_{x\to y} |p(x)-1| \log\frac1{|x-y|} = 0
\]
for every $x,z\in \Omega$ and $y\in Y$. Then
\[
D_\epsilon(u):= \int_\Omega |\nabla u|^{p_{1+\epsilon}(x)} + |u-f|^2\, dx
\] 
$\Gamma$-converges in $L^1(\Omega)$-topology to  
\[
D(u):=V(u, Y) + \int_{\Omega\setminus Y} |\nabla u|^{p(x)} + |u-f|^2\, dx.
\]
\end{thm}

Again, the uniqueness of the 
$\Gamma$-limit allows us to obtain a concrete form of $E$. The proof is again 
similar to Lemma~\ref{cor:GammaDP} and thus skipped.

\begin{cor} 
Suppose that $\Omega$ is a rectangular cuboid and $p$ is strongly 
$\log$-H\"older continuous.
Then
\[
F(u)= V(u, Y) + \int_{\Omega\setminus Y} |\nabla u|^{p(x)} + |u-f|^2\, dx,
\]
where $F$ is defined with the variable exponent function $\phi(x,t):=t^{p(x)}$. 
\end{cor}

%


\end{document}